\documentclass[12pt]{amsart}
\usepackage{enumerate}
\usepackage{etex}
\reserveinserts{28}
\usepackage[latin1]{inputenc}
\usepackage{amsmath}
\usepackage{amsfonts}
\usepackage{amssymb}
\usepackage{mathtools}
\usepackage{color}                    % For creating coloured text and background
\usepackage{hyperref}                 % For creating hyperlinks in cross references
\usepackage[T1]{fontenc}
\newtheorem{thm}{Theorem}[section]
\newtheorem{cor}[thm]{Corollary}
\newtheorem{lem}[thm]{Lemma}
\newtheorem{prop}[thm]{Proposition}

\theoremstyle{definition}
\newtheorem{defn}[thm]{Definition}

\newtheorem{prob}[thm]{Problem}
\newtheorem{example}[thm]{Example}
\numberwithin{equation}{section}

\def\({\left(}
\def\){\right)}

\def\N{\mathbb{ N}}
\def\R{\mathbb{ R}}

\newcommand{\cJ}{{\mathcal J}}

\newcommand{\cI}{\mathcal I}

\def\1{\textbf{1}}

\newcommand{\To}{\longrightarrow}

\def\betaN{\beta \mathbb{N}}

\title[]{Compact spaces associated to separable Banach lattices}

\author[Avil\'es]{Antonio Avil\'es}
\address[Avil\'es]{Universidad de Murcia, Departamento de Matem\'{a}ticas, Campus de Espinardo 30100 Murcia, Spain
	\newline
	\href{https://orcid.org/0000-0003-0291-3113}{ORCID: \texttt{0000-0003-0291-3113} } }
\email{\texttt{avileslo@um.es}}

\author[Mart\'inez-Cervantes]{Gonzalo Mart\'inez-Cervantes}
\address[Mart\'inez-Cervantes]{Universidad de Alicante, Departamento de Matem\'{a}ticas, Facultad de Ciencias, 03080 Alicante, Spain
	\newline
	\href{http://orcid.org/0000-0002-5927-5215}{ORCID: \texttt{0000-0002-5927-5215} } }	
\email{gonzalo.martinez@ua.es}

\author[Rueda Zoca]{Abraham Rueda Zoca}
\address[Rueda Zoca]{Universidad de Murcia, Departamento de Matem\'{a}ticas, Campus de Espinardo 30100 Murcia, Spain
	\newline
	\href{https://orcid.org/0000-0003-0718-1353}{ORCID: \texttt{0000-0003-0718-1353} }}
\email{\texttt{abraham.rueda@um.es}}
\urladdr{\url{https://arzenglish.wordpress.com}}

\author[P. Tradacete]{Pedro Tradacete}
\address[Tradacete]{Instituto de Ciencias Matem\'aticas (CSIC-UAM-UC3M-UCM)\\
	Consejo Superior de Investigaciones Cient\'ificas\\
	C/ Nicol\'as Cabrera, 13--15, Campus de Cantoblanco UAM\\
	28049 Madrid, Spain.}
\email{pedro.tradacete@icmat.es}


\thanks{}

\begin{document}

\begin{abstract}
	We study the class of compact spaces that appear as structure spaces of separable Banach lattices. In other words, we analyze what $C(K)$ spaces appear as principal ideals of separable Banach lattices. Among other things, it is shown that every such compactum $K$ admits a strictly positive regular Borel measure of countable type that is analytic, and in the nonmetrizable case these compacta are saturated with copies of $\beta\mathbb N$. Some natural questions about this class are left open.
\end{abstract}

\maketitle

\section{Introduction}

Given a vector lattice $E$ and a nonzero positive vector $u\in E_+$, there is a natural way to associate a compact space $K_u(E)$: Following \cite{Schaefer72}, $K_u(E)$ can be described as the set of all functions $\varphi:E\To [0,+\infty]$ such that  $\varphi(u) = 1$, $\varphi(x+y)=\varphi(x)+\varphi(y)$, $\varphi(x\wedge y) = \varphi(x)\wedge \varphi(y)$ and $\varphi(rz)  = |r|\varphi(|z|)$ for all $x,y\in E_+$, $z\in E$ and $r\in \R$. Here, $[0,+\infty]$ are the nonnegative reals extended by a point $+\infty$ (satisfying usual conventions plus $+\infty\cdot 0 = 0$), and the topology is the pointwise topology. When $E$ is a Banach lattice, then the principal ideal
$$ E_u= \{y\in E : \exists r>0 : |y|\leq ru\}$$ 
becomes, by the Kakutani representation theorem, canonically isomorphic, as a vector lattice, to the space of continuous functions $C(K_u(E))$. We have a one-to-one Banach lattice homomorphism $$T:(C(K_u(E)),\|\cdot\|_\infty)\To E$$ of norm 1, whose range is $E_u$, and that sends the constant one function to $u$. A vector lattice isomorphism between $C(K)$ and $C(L)$ that preserves the constant functions induces a homeomorphism between $K$ and $L$, so $K_u(E)$ is uniquely determined by the existence of such $T$.  For details, consult \cite{Schaefer72} or \cite[Theorem III.4.5]{Schaeferbook}. The purpose of this paper is to investigate the class of compact spaces that arise in this way from \emph{separable} Banach lattices.

\begin{defn}
	We say that a compact space $K$ is \emph{sick}\footnote{for \emph{separable ideal $C(K)$}} if it is homeomorphic to $K_u(E)$ for some separable Banach lattice and some $u\in E_+$.
\end{defn}

In the sequel, $E$ will always be a Banach lattice. It is shown in \cite[Proposition 2]{Schaefer72} that $K_u(E)$ is homeomorphic to $K_v(E)$ whenever $\overline{E_u} = \overline{E_v}$. The vector $u$ is called a quasi-interior point of $E$ if $E=\overline{E_u}$. Whenever such points exist (and this is the case for all separable Banach lattices), we can talk about $K(E)$, the structure space of $E$, as the compact space homeomorphic to $K_u(E)$ for any quasi-interior point $u\in E$. An important feature is that we can represent $E$ as an ideal of the vector lattice of continuous functions $f:K(E)\To\mathbb{R}\cup\{-\infty,+\infty\}$ with $f^{-1}(\mathbb{R})$ dense, (cf. \cite[Theorem 3.5]{Wickstead92}). The class of sick compacta coincides with the class of structure spaces of separable Banach lattices.\\

Let us now discuss what we have found out about the class of sick compacta. In Section \ref{SectionFirstExamples}, we present the first obvious examples. The structure space of $C(K)$ is the compact space $K$, and therefore all metrizable compact spaces are sick. More generally, for $f\in C(K)_+$, $K_f(C(K))$ is the \v{C}ech-Stone compactification of $\{x : f(x)\neq 0\}$, and therefore $\beta W$ is sick for any open subset $W$ of a compact metric space. A particular case of this is $\beta\mathbb{N}$, which also appears as the structure space of any Banach lattice with the lattice structure given by an unconditional basis, like $c_0$ or $\ell_p$. Finally the structure space of $L_1[0,1]$ is the Stone space $\mathfrak{M}$ of the measure algebra of the Lebesgue measure. Indeed the principal ideal generated by the constant one function is, as a vector lattice, identified with $L_\infty[0,1] = C(\mathfrak{M})$. Thus, $\mathfrak{M}$ gives a nonseparable sick compactum. We also show that the class of sick compacta is not stable under subspaces or quotients, but it contains subspace-universal and quotient-universal spaces.\\

By the way that sick compacta are linked to separable Banach lattices, that are in particular Polish spaces, they should have low complexity in the sense of descriptive set theory. This idea is exploited in Section \ref{SectionAnalytic}. We prove that every sick compactum $K$ admits a strictly positive regular Borel measure $\mu$ of countable type that is moreover \emph{analytic}. This means that, through the natural embedding, $C(K)$ is an analytic subset of $L_1(\mu)$. We also show that for any countable set $D$ of $G_\delta$ points of $K$,  the set of restrictions $\{f|_D : f\in C(K)\}$ form an analytic subset of $\mathbb{R}^D$, and some more general results. These facts are reminiscent of Godefroy's results on Rosenthal compact spaces \cite{Godefroy}. In a little disgression from the main topic, in Section~\ref{SectionRosenthal} we observe that all measures on Rosenthal compacta are analytic. However, the two classes are very different. Both Rosenthal and sick compacta are \emph{of analytic complexity}, but they are orthogonal from the Efimov's problem perspective: while Rosenthal compacta are full of convergent sequences, nonmetrizable sick compacta are full of copies of $\beta\mathbb{N}$. \\

This last statement is made precise in the main result of Section~\ref{SectionCopiesbetaN}: Every sick compactum contains a (possibly trivial) sequence of metrizable closed subsets $K_n$, such that the closure of every discrete sequence $x_n\not\in K_n$ is homeomorphic to $\beta\mathbb{N}$. In particular, nonmetrizable sick compacta always contain copies of $\beta\mathbb{N}$.\\

In Section \ref{SectionSuprema}, we connect sick compacta to a different problem. Given, a sequence $\{e_n\}$ of positive pairwise disjoint vectors in a Banach lattice, we can consider the family $\mathcal{S}_E(\{e_n\})$ of all subsets $A\subseteq\mathbb{N}$ for which the supremum $\sup_{n\in A}e_n$ exists in $E$. What are the families $\mathcal{A}$ of subsets of $\mathbb{N}$ that can be represented as $\mathcal{S}_E(\{e_n\})$ with $E$ separable? The main result of this section is that, if $\mathbb{N}\in \mathcal{A}$, this happens if and only if $\mathcal{A}$ is an algebra of sets containing the finite sets that is isomorphic, as a Boolean algebra, to the clopen subsets of a sick compactum.\\

Borodulin-Nadzieja, Farkas and Plebanek \cite{BFP} have characterized the families that can be represented as $\mathcal{S}_E(\{e_n\})$ with $E$ a Banach lattice obtained from an unconditional basis. These are exactly the so called non-pathological analytic $P$-ideals. This does not give direct information on sick compacta because these families do not contain $\mathbb{N}$ except in one trivial case, and the structure space of an unconditional Banach lattice is always $\beta\mathbb{N}$. However, a variation of their construction obtained by adding an upper bound to the basis (similarly to how $c$ is obtained from $c_0$) provides a new family of examples, described in Section \ref{SectionPideals}. If $\mathcal{I}$ is a non-pathological analytic $P$-ideal, then the Stone space of the algebra generated by $\mathcal{I}$ is a sick compactum.\\

In spite of the number of properties and examples of sick compacta that we have exposed, a satisfactory understanding of this class ---which means understanding the vector lattice structure of separable Banach lattices--- is far from being reached. The paper is scattered with open questions that arise naturally. Perhaps the most fundamental is to have an intrinsic handable topological characterization. We collect some of these questions at the end of the paper.

%One of the main advantages of the representation theorems is that it is possible to transfer all techniques available for constructing elements in $C(K_E)$ to finding certain elements in $E$. Thus, the better we understand the compact space $K_E$, the more handy $E$ becomes. Sometimes a topological property of $K_E$ has a direct translation into a Banach lattice property of $E$; for instance, $K_E$ is Stonian (quasi-Stonian) if and only if $E$ is Dedekind complete (resp.~$\sigma$-Dedekind complete), whereas $K_E$ is hyperstonian if the norm of $E$ is order continuous (see \cite[Theorems 4.14 and 4.15]{Filter94}). Nevertheless, we do not know whether the class of sick compacta (i.e. the class of structure spaces of separable Banach lattices) has a nice characterization in terms of classical topological properties. Despite this fact, we prove that they satisfy some remarkable properties. Namely, in Section \ref{SectionAnalyticMeasures} we prove that every sick compactum admits a strictly positive analytic measure; whereas in Section \ref{SectionCopiesbetaN} it is shown that every nonmetrizable closed subspace of a sick compactum contains copies of $\betaN$.

\section{First examples}
\label{SectionFirstExamples}

\begin{prop}
	Let $W$ be an open subset of a compact metric space $K$, then $\beta W$ is sick.
\end{prop}

\begin{proof}
	Find a continuous function $f:K\To [0,1]$ such that $W=f^{-1}(0,1]$ (for instance a suitable multiple of $f(x):=d(x,K\setminus W)$).  Let $C_b(W)$ be the space of all bounded and continuous functions $\phi:W\To\mathbb{R}$ and let $T:C_b(W)\To C(K)$ be given by 
	$$T\phi(x) = \begin{cases} \phi(x)\cdot f(x) &\text{ if } x\in W, \\ 0 &\text{ if } x\not\in W.\end{cases}$$
	This function $T$ is a vector lattice isomorphism onto its range. We claim that the range of $T$ is exactly the ideal $C(K)_f$ generated by $f$ in $C(K)$. Once we have the claim proved, this ideal would be vector lattice isomorphic to $C_b(W) = C(\beta W)$, which proves that $K_f(C(K))$ is homeomorphic to $\beta W$ and the proof would be finished. 
	
	So let us prove the claim. To do so, if $|\phi|$ is bounded by $\lambda\in (0,+\infty)$, then $|T\phi| \leq \lambda\cdot f$, and this proves that the range of $T$ is contained in $C(K)_f$. Conversely, if $g\in C(K)_f$, then the restriction of $g/f$ to $W$ is a bounded continuous function $\phi\in C_b(W)$ which satisfies that $g=T\phi$.
\end{proof}

In particular, $\betaN$ is a sick compactum. In general, all the examples of the form $\beta W$ are separable. A nonseparable example comes when we look at $L_1[0,1]$. For the constant 1 function $x\in L_1[0,1]$, the ideal $L_1[0,1]_x$ is clearly $L_\infty[0,1]$, which is identified with the space $C(\mathfrak{M})$ where $\mathfrak{M}$ is the Stone space of the measure algebra of the Lebesgue measure on $[0,1]$. Thus $\mathfrak{M}$ is a sick compactum.\\

An elementary remark is that every separable Banach lattice is $ccc$, i.e.~every pairwise disjoint family of positive elements is countable. This property is inherited by ideals, so $C(K)$ is $ccc$ for a sick compactum $K$, and so $K$ itself is $ccc$ (via Urysohn Lemma), in the sense that every pairwise disjoint family of nonempty open sets is countable. We shall see later that something stronger holds: every sick compactum supports a strictly positive measure.
Anyway, we can already notice that the class of sick compacta is not stable under taking subspaces, since $\beta\mathbb{N}$ is sick but $\betaN \setminus \N$ is not because it fails $ccc$ \cite[Theorem 3.22]{Walker}. The class of sick spaces is not stable under continuous images either. We shall see concrete examples later, but we can use a counting argument. There are only $\mathfrak{c}$ many different separable Banach lattices, but there are $2^\mathfrak{c}$ non-homeomorphic separable compact spaces (cf. for instance \cite{Freniche}), all of them continuous images of $\betaN$.\\

Despite this unstability, we do have universal objects. The assignment $E\mapsto K(E)$ takes quotients to inclusions, cf. \cite[Propositions 3]{Schaefer72}. %Every separable Banach lattice can be seen as a sublattice of $C(2^\mathbb{N},L_1)$, the space of continuous functions from the Cantor set into $L_1[0,1]$ \cite{LLOT19}. Therefore its structure space $K(C(2^\mathbb{N},L_1))$ is a sick compactum that maps continuously onto any other one. Similarly, 
Every separable Banach lattice is a quotient of the free Banach lattice $FBL(\mathbb{N})$, which is separable \cite{dPW15}. Thus, the structure space $K(FBL(\mathbb{N}))$ is a sick compactum that contains any other sick compactum as a subspace. The topological nature of this universal sick compactum looks mysterious to us. On the other hand, by Theorem~\ref{analyticmeasure}, every sick compactum will be a continuous image of $\mathfrak{M}$, cf. \cite[415Q-415R]{Fremlinbook}.

\section{Analytic complexity}
\label{SectionAnalytic}

A principal ideal of a Banach lattice $E$ is an $F_\sigma$ subset of $E$. So when $K$ is a sick compactum, $C(K)$ appears bijected with a nice subset of a Polish space $E$. This situation suggests that $C(K)$ should be viewed as an object of low descriptive complexity, but it is not clear at first glance in what sense. In this section, we will present several results that make precise the idea that $C(K)$ is of {analytic} complexity when $K$ is sick. The first such result also encircles our class within the realm of compact spaces that admit strictly positive measures.  By a measure on a compact space $K$ we mean a regular Borel measure on $K$. Such a measure is of countable type if the corresponding space $L_1(\mu)$ is separable, and it is strictly positive if all nonempty open sets have nonzero measure.

\begin{defn}
	A measure $\mu$ on a compact space $K$ will be called \emph{analytic} if it is of countable type and the range of the formal identity $i:C(K)\To L_1(\mu)$ is an analytic set.
\end{defn}

\begin{thm}\label{analyticmeasure}
	If $K$ is sick, then $K$ admits a strictly positive analytic measure.
\end{thm}

\begin{proof}
	Let $T:C(K)\To E$ be the vector lattice isomorphism from $C(K)$ onto an ideal $E_x$ of a separable Banach lattice $E$. Since $E$ is  separable, we have a positive norm-one functional $x^*:E\To \mathbb{R}$ such that $x^*(y)>0$ whenever $y>0$. The composition $x^\ast\circ T$ gives a strictly positive functional on $C(K)$, that is represented by a strictly positive measure $\nu$ on $K$. Let $i:C(K)\To L_1(\nu)$ be the formal inclusion. Note that the natural identity $T(C(K))\To i(C(K))$ is a continuous operator between normed spaces, because 
	$$\|if\| = x^\ast(T|f|) \leq  \|Tf\|$$ for every $f\in C(K)$. Since $$T(C(K))=E_x=\bigcup_{n \in \N}\{y\in E: |y|\leq nx\}$$ is a countable union of closed subsets of a separable Banach lattice $E$, we conclude that $i(C(K))$ is analytic.
\end{proof}

%QUIERO ESTO??
%
%A look at the particular case of Stone spaces of subalgebras of $\mathcal{P}(\mathbb{N})$ may give a taste of the kind of restriction this imposes to sick compacta.
%
%\begin{cor}
%	Let $\mathcal{B}$ be a subalgebra of $\mathcal{P}(\mathbb{N})$ that contains all finite sets. If the Stone space of $\mathcal{B}$ is sick, then $\mathcal{B}$ is an analytic subalgebra of $\mathcal{P}(\mathbb{N})$.
%\end{cor}
%
%\begin{proof}
%	Let $\mu$ be a strictly positive analytic measure on the Stone space $K$ of $\mathcal{B}$, that can be identified with a finitely additive measure on $\mathcal{B}$. We consider the formal identity $i:C(K)\To L_1(\mu)$, the function $\psi:L_1(\mu) \To \mathbb{R}^\mathbb{N}$ given by $\psi(f)(n) = f(n)/\mu\{n\}$. A MEDIO.
%\end{proof}

\begin{cor}\label{coranalytic1}
	If $K$ is sick and $\{x_n : n\in\mathbb{N}\}$ is a set of isolated points of $K$, then $\{(f(x_n))_{n\in\mathbb{N}}: f\in C(K)\}$ is an analytic subset of $\mathbb{R}^\mathbb{N}$.
\end{cor}

\begin{proof}
	Consider the formal identity $i:C(K)\To L_1(\mu)$, the function $j:L_1(\mu)\To \ell_1$ given by $j(f) = (\mu\{x_n\}\cdot f(x_n))_{n}$, and the function $\phi:\ell_1\To \mathbb{R}^\mathbb{N}$ given by 
	$$\phi((t_n)_{n}) =  \Big(\frac{t_n}{ \mu\{x_n\}}\Big)_{n}.$$ Note that $j$ is a bounded operator and $\phi$ is a Borel function, so the composition is Borel. Since $i(C(K))$ is analytic, so is $(\phi\circ j)(i(C(K)))$, which is the set in the statement.
\end{proof}

\begin{cor}\label{coranalytic2}
	Let $\mathcal{B}$ be a subalgebra of $\mathcal{P}(\mathbb{N})$ that contains all finite sets. If the Stone space of $\mathcal{B}$ is sick, then $\mathcal{B}$ must be an analytic subalgebra of $\mathcal{P}(\mathbb{N})$.
\end{cor}

\begin{proof}
	In such a Stone space there is a dense set of isolated points $\{x_n : n\in\mathbb{N}\}$ and, under the natural bijection of $\mathcal{P}(\mathbb{N})$ and $\{0,1\}^\mathbb{N}$, the algebra $\mathcal{B}$ is sent exactly to the set $\{(f(x_n))_{n\in\mathbb{N}}: f\in C(K)\}\cap \{0,1\}^\mathbb{N}$.
\end{proof}

%\begin{prop} If $\mu$ is discrete, with $D = \{x : \mu\{x\}\neq 0\}$, then $\mu$ is analytic if and only if $\{f|_D : f\in C(K)\}$ is an analytic subset of $\mathbb{R}^D$.
%\end{prop}
%
%\begin{proof}
%	We have the formal identity $i:C(K)\To L_1(\mu)$, as well as an isometry $j:L_1(\mu) \To \ell_1(D)$, given by $j(f) = (\mu\{d\}\cdot f(d))_{d\in D}$, and a measurable one-to-one function $\phi:\ell_1(D)\To \mathbb{R}^D$ given by $\phi((x_d)_{d\in D}) =  (x_d / \mu\{d\})_{d\in D}$. We have that $\mu$ is analytic if and only if $i(C(K))$ is analytic, if and only if $\phi(j(i(C(K)))) = \{f|_D : f\in C(K)\}$ is analytic.
%\end{proof}

These results are saying that it is only among compact spaces that have an analytic definition that sick compacta are to be found, while other compacta constructed by set-theoretic techniques should be ruled out. Along the rest of the section, we prove some generalizations of Corollary~\ref{coranalytic1}.

\begin{defn}
	Let $\mathcal{F}_\omega(K)$ be the least family of closed subsets of $K$ that contains the closed $G_\delta$ sets, and is stable under countable intersections and closures of countable unions. Sets in this family will be called \emph{$\sigma$-generated closed} sets.
\end{defn}

\begin{lem}\label{restrictiontosigmaclosedlemma}
	Let $K$ be sick and $T: C(K)\To E_x\subset E$ be a vector lattice isomorphism onto an ideal of a separable Banach lattice. If $L$ is a $\sigma$-generated closed subset of K, then for any $r\geq 0$, the  set $\{Tf : \|f|_L\|_\infty \leq r\}$ is an analytic subset of $E$.
\end{lem}

\begin{proof}
	The first observation is that it suffices to prove that the set $\{Tf : f|_L = 0\}$ is analytic. This is because $\|f|_L\|_\infty \leq r$ if and only if $(|f|-r\mathbf{1})\vee 0|_L=0$, so if we consider the function $\Phi:E_x\To E_x$ given by $\Phi(z) = (|z|-rx)\vee 0$, then
	$$\{Tf : \|f|_L\|_\infty \leq r\} = \Phi^{-1}\{Tf : f|_L = 0\}.$$ Since $E_x$ is an $F_\sigma$-set and $\Phi$ is continuous, the analyticity of $\{Tf : \|f|_L\|_\infty \leq r\}$ follows from that of $\{Tf : \|f|_L\|_\infty =0\}$. 
	
	Let us prove the case when $L$ is a closed $G_\delta$ set. In that case, there exists $g\in C(K)$ such that $g\geq 0$ and $L = g^{-1}(0)$. Then $\{f : f|_L = 0\}$ is the closed ideal generated by $g$ in $C(K)$. So,
	$$\{Tf : f|_L = 0\} = \{y\in E_x : \exists z_1,z_2,\ldots\in E_x : |z_n|\leq nTg, \lim_n\|z_n-y\|_\infty = 0\}.$$
	Note that $\|\cdot\|_\infty$ is a Borel function on the closed set $E_x\subset E$, because its graph is a Borel set. So we conclude that the set $\{Tf : f|_L = 0\}$ is analytic as desired. It is enough to prove now that the family of all closed $L$ such that $\{Tf : f|_L = 0\}$ (and hence also $\{Tf : \|f|_L\|_\infty \leq r\}$) is analytic is stable under countable intersections and closures of countable unions. The latter is obvious because $$\left\{Tf : f|_{\overline{\bigcup_n L_n}} = 0\right\} = \bigcap_n \{Tf : f|_{L_n} = 0\}.$$
	Concerning intersections, let us start with an intersection of just two $L_1\cap L_2$. In this case, the observation is that $\{f : f|_{L_1\cap L_2} = 0\}$ is the closed ideal generated by the set $\{f_1+f_2 : f_1|_{L_1}=0\text{ and }f_2|_{L_2}=0\}.$ The image under $T$ of that set of sums would be analytic, and by the same description as before, the image under $T$ of the closed ideal it generates would also be analytic. From this we deduce that our family is closed under finite intersections. Finally, it is enough to check that the family is stable under decreasing intersections. But if $L_1\supset L_2\supset\cdots$, then we have that
	$$\{f : f|_{\bigcap_n L_n} = 0\} = \left\{f : \forall m\in\mathbb{N}\ \exists n\in\mathbb{N}\ \|f|_{L_n}\|_\infty \leq 1/m \right\}, $$
	and since we assume that each set $A_{n,m} = \{f :\|f|_{L_n}\|_\infty \leq 1/m\}$ is analytic, we conclude that
	$$\{Tf : f|_{\bigcap_n L_n} = 0\} = \{y\in E_x : \forall m\in\mathbb{N}\ \exists n\in\mathbb{N}\ y\in TA_{n,m}\} $$
	is an analytic set.
\end{proof}

\begin{thm}
	Let $\{D_n\}_n$ be a family of $\sigma$-generated closed subsets of a sick compactum. Then the following is an analytic set:
	$$S = \left\{(t_n)\in \mathbb{R}^\mathbb{N} : \exists f\in C(K) : f|_{D_n} \text{ is constantly equal to }t_n\right\}.$$
\end{thm}

\begin{proof}
	Let us start with  the following
	
	\emph{Claim:}	$$ S = \left\{(t_n)\in \ell_\infty : \forall\ p<q\
	\overline{\bigcup\{D_n : t_n\leq p\}} \cap \overline{\bigcup\{D_n : t_n\geq q\}} = \emptyset
	\right\}. $$
	
	\emph{Proof of the claim.} The inclusion $[\subseteq]$ is obvious.  For the converse, suppose that $(t_n)$ satisfies the property above to be in the right-hand side set $S'$. For each $x\in \overline{\bigcup_n D_n}$, the set $$R_x : = \bigcap\left\{ [a,b] : x\in \overline{\bigcup_{t_n\geq a}D_n} \cap  \overline{\bigcup_{t_n\leq b}D_n}\right\}$$ is a singleton. Indeed, for an intersection of closed intervals to be empty we would need to have among them at least two disjoint intervals $[a,p]$ and $[q,b]$ with $p<q$, and we would get that $x\in \overline{\bigcup\{D_n : t_n\leq p\}} \cap \overline{\bigcup\{D_n : t_n\geq q\}} = \emptyset$, a contradiction. We cannot have two elements $p<q$ inside $R_x$. Otherwise, assume that $\{t_n\}$ is included inside the interval $[a,b]$ and $p<p'<q'<q$, then we will have that $x\in \overline{\bigcup_n D_n} = \overline{\bigcup_{t_n\leq q'}D_n}\cup \overline{\bigcup_{t_n\geq p'}D_n}$. Say, for instance, that $x\in \overline{\bigcup_{t_n\leq q'}D_n}$. Then $x\in\overline{\bigcup_{t_n\geq a}D_n}\cap\overline{\bigcup_{t_n\leq q'}D_n}$, which contradicts that $q\in R_x$. Similarly if $x\in \overline{\bigcup_{t_n\geq p'}D_n}$, it contradicts that $p\in R_x$. Thus, we have a well defined function $f:\overline{\bigcup_{n}D_n}\To \mathbb{R}$ in such a way that $\{f(x)\}=R_x$. It is clear that $f|_{D_n}$ is constant equal to $t_n$. By Tietze's extension theorem, it is enough to check that $f$ is continuous. The argument follows a similar line as before, if $f(x)=p<q$ and $p<p'<q'<q$ then
	$$x\not\in \overline{\bigcup_{t_n\geq p'}D_n} \Rightarrow x\in \overline{\bigcup_{t_n\leq q'}D_n} \Rightarrow f(x)\leq q' \Rightarrow f(x)<q,$$
	so $K\setminus \overline{\bigcup_{t_n\geq p'}D_n}$ is a neighborhood of $x$ where $f(x)<q$. An analogous argument works for the opposite inequality.

	In the description of $S$ in the above Claim one may demand that $p$ and $q$ are rational. So it is enough to show that for all $p<q$ the set
	$$ S_{p,q}= \left\{(t_n)\in \mathbb{R}^\mathbb{N} : \
	\overline{\bigcup\{D_n : t_n\leq p\}} \cap \overline{\bigcup\{D_n : t_n\geq q\}} = \emptyset
	\right\} $$
	is analytic. For this, it is enough to show that
	$$ S' = \left\{(A,B)\in \mathcal{P}(\mathbb{N})\times \mathcal{P}(\mathbb{N})  : \
	\overline{\bigcup\{D_n : n\in A\}} \cap \overline{\bigcup\{D_n :n\in B\}} = \emptyset
	\right\} $$
	is analytic, because $S_{p,q}$ is the preimage of $S'$ under the measurable assignment $$(t_n) \mapsto (\{n : t_n\leq p\},\{n : t_n\geq q\}).$$ But $S'$ can be written as
	$$S' = \left\{(A,B)  : \ \exists f\in C(K) \ f|_{D_n}=0 \text{ for }n\in A, (1-f)|_{D_n} = 0 \text{ for }n\in B \right\}. $$
	This is analytic by Lemma~\ref{restrictiontosigmaclosedlemma}.
\end{proof}

We call a point $x\in K$ $\sigma$-generated if $\{x\}$ is a $\sigma$-generated closed set. Every $G_\delta$ point is $\sigma$-generated. %Also, the limit of a convergent sequence of $\sigma$-generated points is a $\sigma$-generated point, because $$\left\{\lim_n x_n\right\} = \bigcap_n \overline{\{x_m : m\geq n\}}.$$ Hence all points in the sequential closure of the $G_\delta$ points are $\sigma$-generated.

\begin{cor}\label{analyticrestrictionGdeltapoints}
	If $D$ is a countable set of $\sigma$-generated points of a sick compactum $K$, then $\{f|_D : f\in C(K)\}$ is an analytic subset of $\mathbb{R}^D$.
\end{cor}

%By a theorem of Godefroy, a separable compact space $L$ is Rosenthal if and only if $\{f|_D : f\in C(K)\}$ is an analytic subset of $\mathbb{R}^D$ for every (dense) countable $D\subset L$.

Notice that Corollary~\ref{analyticrestrictionGdeltapoints} does not hold if the assumption on the points of $D$ is removed. The compact space $\beta\mathbb{N}$ is sick. If we fix a nonprincipal ultrafilter $\mathcal{U}$, then it is a non $G_\delta$-point and taking $D=\mathbb{N}\cup \{\mathcal{U}\}$, the set $\{f|_D : f\in C(K)\} = \{(x_1,x_2,\ldots,x_\mathcal{U}) : x_\mathcal{U} = \lim_\mathcal{U} x_n\}$ is not analytic by the classical result of Sierpi\'{n}ski \cite{nonmeasurableultrafilter} that nonprincipal ultrafilters are not measurable. We remind an argument: if $\mathcal{U}$ is measurable in the product probability space $\{0,1\}^\mathbb{N}$, then it should have probability $1/2$ since its complement coincides with its $0-1$ switch, but by Kolmogorov's zero-one law its probability must be either $0$ or $1$. \\

%I am not sure how much generality is really achieved when considering $\sigma$-generated closed sets instead of $G_\delta$-closed in this context:
%\begin{prob}
%	In a sick compactum, can we find a sequence of $G_\delta$ closed sets whose union has non $G_\delta$ closure? Can we find a sequence of $G_\delta$ points that converges to a non $G_\delta$ point? Lo último no es posible por el Teorema de Gonzalo mejorado a cerrados Gdelta.
%\end{prob}

%This settles what kind of chain conditions sick compacta must satisfy. But this is far away from a characterization. There are $2^\mathfrak{c}$ non-homeomorphic separable compact spaces (any separable compactum admits a strictly positive measure of countable type), but there are only $\mathfrak{c}$ many homeomorphism classes of sick compacta.\\

\section{A disgression: Analytic measures on Rosenthal compacta}\label{SectionRosenthal}

The results of the previous section remind us of Godefroy's characterization of Rosenthal compacta \cite{Godefroy}: A separable compact space $K$ is a Rosenthal compactum if and only if $\{f|_D : f\in C(K)\}$ is an analytic subset of $\mathbb{R}^D$ for all countable dense subsets $D$ of $K$. The word \emph{dense} may be removed since the property of being Rosenthal is inherited by closed subspaces. Another fact first claimed by Bourgain, with proofs given by Todorcevic \cite{Tod1BC}, Marciszewski and Plebanek \cite{MarPle} and Plebanek and Sobota \cite{PlebSob} is that all measures on a Rosenthal compactum are of countable type. The following puts together both results.

\begin{thm}
	Every measure $\mu$ on a Rosenthal compactum $K$ is analytic. A separable compactum $K$ is Rosenthal if and only if all measures on $K$ are analytic.
\end{thm}

\begin{proof}
	We start proving the first statement. We take a measure $\mu$ on $K$ that we know that is of countable type as have been mentioned above. By \cite[Proposition 8]{Godefroy} one can assume that the compact space $K$ is separable. By \cite[Theorem 9]{Godefroy}, we have that $(C(K),{\rm cyl}_{C(K)})$ is Suslinean. Here, the cylindrical $\sigma$-algebra ${\rm cyl}_X$ on a Banach space $X$ is the least $\sigma$-algebra that makes all bounded linear functionals measurable. A measurable space is Suslinean if it is isomorphic to an analytic space with its Borel $\sigma$-algebra. Look at the formal identity $i:C(K)\To L_1(\mu)$. Since it is a bounded operator, it is ${\rm cyl}_{C(K)}$-to-${\rm cyl}_{L_1(\mu)}$-measurable. But in every separable Banach space, the cylindrical and the Borel $\sigma$-algebra coincide. So, $i(C(K))$, endowed with the Borel structure inherited from $L_1(\mu)$ is the image of a Suslinean space under a measurable function. We conclude that $i(C(K))$ is analytic as desired. Finally, we look at the second statement. Suppose that $K$ is a separable compact space where all measures are analytic. Let $D=\{x_n : n=1,2,3,\ldots\}$ be any countable subset of $K$. Consider a discrete probability measure on $K$ with $\mu\{x_n\} = 2^{-n}$, that must be analytic. Reproducing the proof of Corollary~\ref{coranalytic1}, this implies that $\{f|_D : f\in C(K)\}$ is an analytic subset of $\mathbb{R}^D$. By Godefroy's characterization, we just proved that $K$ is Rosenthal.
\end{proof}

%All this may suggest a close relation between sick and Rosenthal compacta. However, as we shall see, the two classes are rather orthogonal to each other.

\section{Omnipresence of $\beta\mathbb{N}$}
\label{SectionCopiesbetaN}

\begin{thm}
	\label{TheoDecompositionSickMetrizbleBetaN}
	Let $K$ be a sick compactum. There exists an increasing sequence $K_n\subseteq K$ of metrizable closed (possibly empty) subspaces of $K$ such that if $\{x_n\}_n$ is a discrete sequence of points in $K$ with $x_n \notin K_n$ for every $n \in \N$, then $\overline{\{x_n\}_n}$ is a copy of $\betaN$. 
\end{thm}
\begin{proof}
	Let us fix a sick compactum $K$ and let $T:C(K)\To E$ be a vector lattice isomorphism of $C(K)$ onto a principal ideal $E_x$ of a separable Banach lattice $E$. We can assume that $\|x\|=1$. 
	We define the function $N:K \To [0,1]$ by the formula
	$$ N(t):= \inf \{\|Tf\|: f\in B_{C(K)},~f(t)=1\} \mbox{ for every } t\in K.$$
	\emph{Claim 1.} For any $\delta>0$, the set $$\widehat{K}_\delta=\{t\in K: N(t)\geq \delta \}$$
	is a metrizable closed subspace of $K$.\\
	
	\noindent
	\emph{Proof of the claim:}
	We prove first that $\widehat{K}_\delta$ is closed.
	Let $t \notin \widehat{K}_\delta$, i.e. such that $N(t)<\delta$. Take $\varepsilon =\frac{\delta-N(t)}{2}$ and fix $f\in B_{C(K)}$ such that $f(t)=1$ and $\|Tf\|<N(t)+\varepsilon$. Without loss of generality we can assume $f\geq0$. Let 
	$$U=\Big\{s\in K: f(s)>\frac{N(t)+\varepsilon}{\delta}\Big\}.$$ 
	Since $\frac{N(t)+\varepsilon}{\delta}<1$ and $f$ is continuous, $U$ is an open neighborhood of $t$. Let us show that $U\cap \widehat{K}_\delta =\emptyset$.
	Pick $s \in U$. The function  $g:=(\frac{1}{f(s)}f)\wedge 1 \in B_{C(K)}$ satisfies $g(s)=1$ and 
	$$\|Tg\| \leq \frac{1}{f(s)}\|Tf\| \leq \frac{1}{f(s)}(N(t)+\varepsilon) < \delta, $$
	so $s\notin \widehat{K}_\delta$ and $U\cap \widehat{K}_\delta =\emptyset$ as desired.

	We check now that $\widehat{K}_\delta$ is metrizable. Set
	$$ \cI = \{h \in C(K): h(s)=0 \mbox{ for every }s\in \widehat{K}_\delta \}.$$
	Note that $C(\widehat{K}_\delta)$ is isomorphic to $C(K)/\cI$. Furthermore, if we take $\cJ=\overline{T(\cI)} \subseteq E$, then $\cJ$ is a closed ideal of $E$ and $T$ induces an operator $\hat{T}:C(K)/\cI \To E/\cJ$ through the formula $\hat{T}\left({h}+\cI\right)={Th} + \cJ$.
	We claim that $\hat{T}$ is an isomorphism, so the metrizability of $\widehat{K}_\delta$ will follow from the separability of $E/\cJ$.
	Take any ${h} \in C(K)$ with $\| {h}+\cI\|=1$. We are going to show that $\|Th+\cJ\| \geq \delta $ and therefore $\hat{T}$ is an isomorphism. 
	Suppose by contradiction that $\|{T}h+\cJ\| <\delta $. By definition of $\cJ$, this is equivalent to the existence of a function $g\in \cI$ such that $\|Th-Tg\|=\|T(h-g)\|<\delta$.
	On the one hand, since $\|{h}+\cI\|=1$, we have that $|h(t)|\leq 1$ for every $t\in \widehat{K}_\delta$ and there exists $s\in \widehat{K}_\delta$ such that $|h(s)|=1$.
	On the other hand, it follows from the definition of $\cI$ that $g(s)=0$, so $|(h-g)(s)|=1$. Let $f:=|h-g|\wedge 1$. Then $f(s)=1$, $f \in B_{C(K)}$ and 
	$$\|Tf\| \leq \|T|h-g|\|=\|T(h-g)\|<\delta.$$ 
	But then $s \notin \widehat{K}_\delta$, which yields a contradiction and finishes the proof of Claim 1.
	\medskip
	
	Set now $K_n:=\widehat{K}_{1/{2^n}}$ for every $n\in \N$. It remains to show that if $\{x_n\}_n$ is a discrete sequence of points in $K$ with $x_n \notin K_n$ for every $n\in \N$ then $\overline{\{x_n\}_n}$ is a copy of $\betaN$. This is an immediate consequence of the following.\\
	
	\noindent
	\emph{Claim 2.} If $\{x_n\}_n$ is a discrete sequence of points in $K$ with $\sum_n N(x_n) < \infty$, then $\overline{\{x_n: n \in \N\}}$ is a copy of $\betaN$. \\
	
	\noindent
	\emph{Proof of the claim:} By definition of the funcion $N$, there exists a sequence of norm-one functions $\{f_n\}_n$ in $B_{C(K)}$ such that $f_n(x_n)=1$ and $\|Tf_n\|<N(x_n)+\frac{1}{2^n}$. Furthermore, since $\{x_n\}_n$ is discrete, we can take pairwise disjoint functions $g_n \in B_{C(K)}$ such that $g_n(x_n)=1$. Set $h_n=|g_nf_n|$ for every $n\in \N$. Then $\{h_n\}_n$ is a sequence of pairwise positive disjoint functions such that $h_n(x_n)=1$ and $$\|Th_n\|\leq \|Tf_n\|<N(x_n)+\frac{1}{2^n}.$$ 
	Since $\sum_n (N(x_n)+\frac{1}{2^n})<\infty$, all subsequences of $\{Th_n\}_n$ have a supremum, which indeed belongs to $E_x=T(C(K))$ because $Th_n \leq x$ for every $n\in \N$, so all subsequences of $\{h_n\}$ have a supremum in $B_{C(K)}$. 
	But this implies that any bounded function on $\{x_n:n\in \N\}$ extends to a continuous function on $\overline{\{x_n: n \in \N\}}$, so $\overline{\{x_n: n \in \N\}}$
	becomes a copy of $\beta\mathbb{N}$. 
\end{proof}

\begin{cor}
	\label{CorobetaNSick}
	If $K$ is a sick compactum, then every nonmetrizable closed subspace contains a copy of $\betaN$.
\end{cor}
\begin{proof}
	By Theorem \ref{TheoDecompositionSickMetrizbleBetaN}, we can write $K=L \cup \bigcup_n K_n$, where $K_n$ is an increasing sequence of metrizable closed subspaces of $K$ and such that if $\{x_n\}_n$ is a discrete sequence of points with $x_n \notin K_n$ for every $n\in \N$ then $\overline{\{x_n: n \in \N\}}$ is homeomorphic to $\beta\mathbb{N}$.
	Thus, if $S\subseteq K$ is an infinite closed subspace of $K$, then either $S \setminus K_n$ is finite for some $n\in \N$ and therefore $S$ is metrizable, or else $S$ contains a discrete sequence of points  with $x_n \notin K_n$ for every $n\in \N$ as desired.
\end{proof}

As another application, we show that products of sick compacta are not sick in general.

\begin{cor}\label{corproductos}
	If $K\times L$ is sick then either $K$ or $L$ is a metrizable compactum.
\end{cor}

\begin{proof}
	Consider the metrizable closed subspaces $Z_n$ that Theorem~\ref{TheoDecompositionSickMetrizbleBetaN} gives for $K\times L$. Let $K_n$ and $L_n$ be the projections of $Z_n$ onto the first and second coordinates respectively. If neither $K$ nor $L$ are metrizable, then $K\neq K_n$ and $L\neq L_n$ for all $n$, so we can find discrete sequences $\{x_n\}$ and $\{y_n\}$ with $x_n\in K\setminus K_n$ and $y_n\in L\setminus L_n$. Consider any function $\alpha:\mathbb{N}\To \mathbb{N}$ with $\alpha (n)>n^2$ for all $n$, and look at the following sequence of points in $K\times L$:
	For $n\in\mathbb N$ and $1\leq k\leq 2n-1$, let 
	$$z_{(n-1)^2+k}=\left\{
	\begin{array}{ll}
	(x_{\alpha(k)},y_{\alpha(n)}) &   \text{ for } 1\leq k\leq n,  \\
	& \\
	(x_{\alpha(n)},y_{\alpha(2n-k)}) &   \text{ for } n< k\leq 2n-1.
	\end{array}
	\right.$$
	%	$\begin{array}{l}z_1 = (x_{\alpha 1},y_{\alpha 1}),\\ z_2 = (x_{\alpha 1},y_{\alpha 2}), z_3 = (x_{\alpha 2},y_{\alpha 2}), z_4 = (x_{\alpha 2},y_{\alpha 1}),\\z_5 = (x_{\alpha 1},y_{\alpha 3}),z_6 = (x_{\alpha 2},y_{\alpha 3}), z_7 = (x_{\alpha 3},y_{\alpha 3}), z_8 = (x_{\alpha 3},y_{\alpha 2}),z_9 = (x_{\alpha 3},y_{\alpha 1}),\ldots
	%		\end{array}$
	
	This is a discrete sequence of points with $z_n\not\in Z_n$, so its closure should be homeomorphic to $\betaN$. But this closure equals $\overline{\{x_{\alpha (n)}\}} \times \overline{\{y_{\alpha (n)}\}} $ while $\betaN$ is not the product of two infinite compact spaces.
\end{proof}

We conjecture that if the product of two infinite compacta $K\times L$ is sick, then \emph{both} $K$ and $L$ must be metrizable. We shall give some partial results supporting this in Section~\ref{SectionSuprema}.

\section{Algebras of sets associated to sick compacta}\label{SectionSuprema}

So far, we know that sick compacta support strictly positive measures, are of analytic complexity in the appropriate sense and they are full of copies of $\beta\mathbb{N}$ wherever they are not metrizable. But, as we shall see, even having all those features is not enough to be sick. In order to get a deeper insight into the class, we are going to focus now on the following class.

\begin{defn}
	An algebra $\mathcal{A}\subset \mathcal{P}(\mathbb{N})$ will be called an \textit{s-sick algebra} if it contains all finite sets and is isomorphic to the algebra of clopen sets of a sick compactum.
\end{defn}

One motivation to look at this class is that it characterizes the families of subsequences of a disjoint sequence with a supremum, whose suprema exist in a separable Banach lattice.

\begin{thm}\label{supremaofsequences}
	For $\mathcal{A}\subseteq\mathcal{P}(\mathbb{N})$ the following are equivalent:
	\begin{enumerate}
		\item $\mathcal{A}$ is an s-sick algebra.
		\item There exists a separable Banach lattice $Y$ and a sequence of pairwise disjoint positive elements $\{e_n\}_n\subset Y$ with a supremum such that
		$$\mathcal{A} = \left\{A\subseteq\mathbb{N} : \exists \sup_{n\in A}e_n \in Y\right\}$$
		\item There exists a separable Banach lattice $Y$ with  atoms $\{e_n\}_n\subset Y$ that have a supremum and such that
		$$\mathcal{A} = \left\{A\subseteq\mathbb{N} : \exists \sup_{n\in A}e_n \in Y\right\}$$
	\end{enumerate}
\end{thm}

\begin{proof}
	$[1\Rightarrow 2]$ Let $K$ be a sick compactum together with a Boolean isomorphism $\phi:\mathcal{A}\To {\rm clopen}(K) = C(K,\{0,1\})$, and let $T:C(K)\To Y$ be a lattice homomorphism onto $Y_e$ that takes the constant one function to $e$. The elements $e_n = T\phi\{n\}$ satisfy the required conditions. 
	
	$[2\Rightarrow 3]$ Consider $Y'$ the closed sublattice of $Y$ generated by $\{\sup_{i\in A} e_i:A\in\mathcal A\}$. We claim that $(e_n)$ are atoms in $Y'$. Indeed, suppose $y\in Y'$ satisfies $0\leq y\leq e_n$ for some $n\in\mathbb N$. Note that $y$ is a limit of a sequence $(y_k)\subset Y'$ such that each $y_k$ belongs to the sublattice generated by finitely many elements of the form $(\sup_{i\in A_j^k} e_i)_{1\leq j\leq p_k}$ where $A^k_1,\ldots, A^k_{p_k}\in\mathcal A$. Without loss of generality we can assume $0\leq y_k\leq e_n$ for every $k\in\mathbb N$.
	
	For each $k\in\mathbb N$ we can find a finite collection of pairwise disjoint $B^k_1,\ldots,B^k_{q_k}\in\mathcal A$ so that the subalgebra of $\mathcal A$ generated by $\{A^k_1,\ldots, A^k_{p_k}\}$ coincides with the subalgebra generated by $\{B^k_1,\ldots,B^k_{q_k}\}$. Since the sequence $(\sup_{i\in B^k_j}e_i)_{j=1}^{q_k}$ are pairwise disjoint it follows that $y_k$ belongs to the linear span of $(\sup_{i\in B^k_j}e_i)_{j=1}^{q_k}$. In this situation, the inequality $0\leq y_k\leq e_n$ implies that $y_k=\lambda_k e_n$ for some $0\leq \lambda_k\leq 1$. Taking the limit in $k$ it follows that $y=\lambda e_n$ for some $0\leq\lambda\leq1$. Hence, $(e_n)_{n\in\mathbb N}$ are atoms in $Y'$ as claimed.

	$[3\Rightarrow 1]$ Let $e=\sup_n e_n$ and let $K$ be the sick compactum for which there is a vector lattice isomorphism $T:C(K)\To Y_e$ that takes the constant one function to $e$. Since $e_n$ is an atom of $Y$, it follows that $T^{-1}e_n$ is an atom of $C(K)$. Atoms of a space of continuous functions are positive multiples of characteristic functions of isolated points, so $T^{-1}e_n = \lambda_n 1_{\{p_n\}}$. Since $\sup_ n e_n = e$, we have that $\sup_n T^{-1}e_n$ is the constant one function. This implies that $T^{-1}e_n = 1_{\{p_n\}}$ for all $n$. Now consider the function $\Psi:{\rm clopen}(K)\To \mathcal{P}(\mathbb{N})$ given by $\Psi(B) = \{n : p_n\in B\}$. We claim that $\Psi$ is a Boolean isomorphism onto $\mathcal{A}$. It is clear that $\Psi$ is a Boolean homomorphism. It is one-to-one because the points $\{p_n\}$ are dense in $K$, as $\sup_n 1_{\{p_n\}}$ is the constant one function. The range of $\Psi$ is $\mathcal{A}$, because $f = \sup_{n\in A}1_{\{p_n\}}$ in $C(K)$ if and only if $f$ is the characteristic function of $\overline{\{p_n : n\in A\}}$ and this set is clopen. 	
\end{proof}

It should be noted that we do not require in the definition of s-sick algebra that the compact space $K$ is totally disconnected. The space $K$ that appears in the proof of $[3\Rightarrow 1]$ above may not be totally disconnected. It seems like an elementary question, but we do not know if the Stone space of an s-sick algebra is a sick compactum.\\

Our next observation is, roughly speaking, that the elements of an s-sick algebra $\mathcal{B}$ that are hereditarily in $\mathcal{B}$ form a substantial part of $\mathcal{B}$. We first introduce the following terminology. Given $\mathcal{B}$, let us define  
$$J(\mathcal{B}) = \{A\subset\mathbb{N} : \mathcal{P}(A)\subseteq\mathcal{B}\},$$ 
which is an ideal of $\mathcal{P}(\mathbb{N})$. For $J\subseteq\mathcal{P}(\mathbb{N})$, set 
$$J^\perp = \{A\subseteq \mathcal{P}(\mathbb{N}) : \forall B\in J\ A\cap B\text{ is finite}\}.$$
\medskip

\begin{prop}\label{orthogonalideal}
	If $\mathcal{B}\subseteq\mathcal{P}(\mathbb{N})$ is an s-sick algebra, then $J(\mathcal{B})^\perp$ is a countably generated ideal of $\mathcal{P}(\mathbb{N})$, and $J(\mathcal{B})^\perp\cap \mathcal{B}$ is countable.
\end{prop}

\begin{proof}
	Let $\{x_n\}$ be the associated sequence of atoms in a separable Banach lattice given by Theorem \ref{supremaofsequences}. We claim that 
	$$J(\mathcal{B})^\perp = \{A\subseteq\mathbb{N} : \inf_{n\in A}\|x_n\| > 0\},$$ 
	which is countably generated by the sets $A_m = \{n\in \mathbb{N} : \|x_n\|>1/m\}$. 
	
	For the inclusion $[\subseteq]$ suppose that $\inf_{n\in A}\|x_n\|=0$. Then there exists an infinite set $A'\subset A$ such that $\sum_{n\in A'}\|x_n\|<+\infty$. But this means that $A'\in J(\mathcal{B})$, and therefore $A\not\in J(\mathcal{B})^\perp$. 
	
	For the inclusion $[\supseteq]$, we must prove that if $\inf_{n\in A}\|x_n\|=\varepsilon>0$ and $B\in J(\mathcal{B})$ then $C = A\cap B$ is finite. If not, $\mathcal{P}(C)\subseteq \mathcal{B}$ and for all $D\subseteq C$ there exists $x_D = \sup_{n\in D}x_n$. If $n\in D\setminus D'$, then $x_D - x_{D'} \geq x_n$. Therefore for any different $D,D'\in \mathcal{P}(C)$ we have $\|x_D-x_{D'}\|\geq \varepsilon$. Since there are uncountably many subsets of $C$, this contradicts that all those vectors belong to a separable $E$. 
	
	The same argument proves the last statement.
\end{proof}

As an application, we provide some partial results, that add to Corollary~\ref{corproductos}, supporting the conjecture that a product of two infinite compact spaces is sick only when both factors are metrizable. 

\begin{prop}\label{propproductcountablyclopens}
	If $K\times L$ is sick and has a dense countable set of isolated points, then 
	both $K$ and $L$ have countably many clopen sets.
\end{prop}

\begin{proof}
	Let $fin\subset \mathfrak{A},\mathfrak{B}\subseteq \mathcal{P}(\mathbb{N})$ be the algebras of clopens of $K$ and $L$ respectively, and let $\mathfrak{A}\otimes\mathfrak{B}\subseteq \mathcal{P}(\mathbb{N}\times\mathbb{N})$ the product algebra, that represents the algebra of clopens of the product. Here recall that $\mathfrak{A}\otimes\mathfrak{B}$ is the algebra generated by the sets of the form $A\times B$ with $A\in \mathfrak{A}$ and $B\in \mathfrak{B}$. If $\mathfrak{A}\otimes\mathfrak{B}$ is s-sick, then by Proposition~\ref{orthogonalideal}, $J(\mathfrak{A}\otimes\mathfrak{B})^\perp$ is countably generated.  Note that $$J(\mathfrak{A}\otimes\mathfrak{B}) = \{(A\times F) \cup (G \times B) : A\in J(\mathfrak{A}),\ B\in J(\mathfrak{B}),\ F,G\text{ finite}\},$$
	$$J(\mathfrak{A}\otimes\mathfrak{B})^\perp= \{X : \forall n\ \{k: (k,n)\in X\}\in J(\mathfrak{A})^\perp, \{k : (n,k)\in X\}\in J(\mathfrak{B})^\perp\}.$$
	We claim that the only way that the latter ideal can be countably generated is that $J(\mathfrak{A})^\perp=J(\mathfrak{B})^\perp = \mathcal{P}(\mathbb{N})$. Indeed, assume for instance that there exists a set $D\not\in J(\mathfrak{A})^\perp$ and $X_1,X_2,\ldots$ are generators of $J(\mathfrak{A}\otimes\mathfrak{B})^\perp$. For every $n$ pick a different $d_n\in D$ such that $(d_n,n)\not\in X_n$. The set $\{(d_n,n) : n\in\mathbb{N}\}\in J(\mathfrak{A}\otimes\mathfrak{B})^\perp$ but it is not contained in any $X_n$, and this is a contradiction. By the last statement of Proposition~\ref{orthogonalideal} we conclude that $\mathfrak{A}$ and $\mathfrak{B}$ are countable.
	%We conclude that $J(\mathfrak{A})=J(\mathfrak{B})=fin$. This means that no subsequence of the dense set of isolated points of either $K$ or $L$ gives a clopen copy of $\beta\mathbb{N}$.
\end{proof}

\begin{cor}
	$\beta\mathbb{N}\times (\{1/n\}_n\cup\{0\})$ is not sick.
\end{cor}

Every s-sick algebra $\mathcal{B}$ is analytic, and $J(\mathcal{B})$ is a $\mathbf{\Pi}^1_2$ set. We do not know if $J(\mathcal{B})$ must actually be an analytic ideal. What we know is that not all analytic ideals may appear as $J(\mathcal{B})$. One reason is that, if we think of the sequence $\{e_n\}$ in Theorem~\ref{supremaofsequences}, all subsequences whose norms are summable belong to $J(\mathcal{B})$ so, in the nontrivial cases, $J(\mathcal{B})$ must contain summable ideals. Here is a concrete example: 

\begin{example}
	Let $\mathcal{I}$ be the ideal of subsets of $\mathbb{Q}_1 = \mathbb{Q}\cap [0,1]$ with a finite number of accumulation points. Then $\mathcal{I}\neq J(\mathcal{B})$ for any s-sick algebra $\mathcal{B}$ of subsets of $\mathbb{Q}_1$. In particular, the algebra made of $\mathcal{I}$ and the complements of $\mathcal{I}$ is not s-sick, and its Stone space is not a sick compactum.
\end{example}

Assume that $\mathcal{I} = J(\mathcal{B})$ and that $\mathcal{B}$ represents the subsequences of $\{x_q : q\in\mathbb{Q}_1\}$ with a supremum as in Theorem~\ref{supremaofsequences}. Since $\mathcal{I}^\perp$ are just the finite sets, if follows from the proof of Proposition~\ref{orthogonalideal} that $\lim_q \|x_q\| = 0$. Enumerate all rational open intervals of $(0,1)$ as $J_1,J_2,\ldots$, and choose by induction $q_n\in J_n\cap \mathbb{Q}_1$, all different such that $\|x_{q_n}\|< 2^{-n}$. Then $\sum_{n} \|x_{q_{n}}\|<+\infty$ but $\{q_{n} : n\in\mathbb{N}\}\not\in J(\mathcal{B})$, a contradiction.

\section{Non-pathological analytic P-ideals}\label{SectionPideals}

Non-pathological analytic $P$-ideals were introduced by Farah \cite{Farah} as those ideals that are of the form $\text{Exh}(\varphi)$ for some lower semicontinuous submeasure $\varphi$ that is the supremum of those measures that it dominates. After decoding what all these words mean, one can rephrase the concept into a working definition like this:

\begin{defn}
	An ideal $\mathcal{I}$ of subsets of $\mathbb{N}$ is a non-pathological analytic $P$-ideal if and only if there exists a set $C\subseteq c_{00}\cap B_{\ell_1}^+$ of finitely supported positive elements of the ball of $\ell_1$ such that
	$$\mathcal{I} = \mathcal{I}(C) = \left\{A\subseteq \mathbb{N} : \lim_m \sup_{c\in C}\sum_{n\in A, n\geq m} c_n = 0 \right\}$$
\end{defn}

It follows from results of Borodulin-Nadzieja, Farkas and Plebanek, cf. \cite[Theorem 5.3]{BFP} and \cite[Proposition 5.3]{BF}, that an ideal that contains the finite sets is a non-pathological analytic $P$-ideal if and only if there exists an unconditional basis $\{e_n\}$ of a Banach space such that $\mathcal{I} = \{A\subset \mathbb{N} : \sum_{n\in A}e_n\text{ converges unconditionally}\}$. When we look at a space with unconditional basis as a Banach lattice, the unconditional convergence of the series is the same as the existence of the supremum. This reminds us of what we studied in Section~\ref{SectionSuprema}, with the difference that there we assumed that the global supremum $\sup_n e_n$ exists, which does not occur inside the span of the basis out of trivial cases. What we are going to do is just to add this supremum to the space from \cite{BFP} to get the main result of the section:

\begin{thm} If $\mathcal{I}$ is a non-pathological analytic $P$-ideal that contains the finite sets, then the Stone space of the algebra generated by $\mathcal{I}$ is a sick compactum.
\end{thm}

\begin{proof}
	Observe that the algebra generated by $\mathcal{I}$ is of the form $\mathcal{I}\cup \mathcal{F}$ where $\mathcal{F}$ is the filter made of the complements of sets in $\mathcal{I}$. Let $C \subseteq c_{00}\cap B_{\ell_1}^+$ be the set such that $\mathcal{I} = \mathcal{I}(C)$. Note that, since the finite sets belong to $\mathcal{I}$, the ideal $\mathcal{I}(C)$ does not change if we add to $C$ the sequences $(1,0,0,0,\ldots)$, $(0,1/2,0,0,0,\ldots)$, $(0,0,1/3,0,0,\ldots)$ and so on. So we can suppose that
	$\alpha_n = \sup\{c_n : c\in C\}>0$ for every $n$. We shall define now three Banach lattices $X_0$, $X_1$ and $X_\infty$. In the terminology of \cite{BF}, $X_0 = \text{EXH}(\Phi)$ and $X_\infty = \text{FIN}(\Phi)$, while $X_1$ will be the Banach lattice that we are looking for.

	For a sequence $x=(x_n)$ of real numbers, define the norm
	
	$$\|x\| = \sup\left\{\sum c_n |x_n| : c\in C\right\}\leq \|x\|_\infty$$

	Consider first $X_\infty = \{x\in\mathbb{R}^\mathbb{N} : \|x\|<\infty\}$. This is a Banach lattice. The lattice norm properties are easily checked, only completeness requires a little thought. If $\{x^p : p\in\mathbb{N}\}$ is a Cauchy sequence, then it pointwise converges to some $x\in\mathbb{R}^\mathbb{N}$ (because $\alpha_n|y_n|\leq \|y\|$). The Cauchy sequence must be bounded and what we want to prove is that $\|x\|\leq \sup_p\|x^p\|$. So we take $c\in C$ and we want to check that $\sum c_n|x_n| \leq \sup_p\|x^p\|$. But since $c$ is finitely supported $$\sum c_n |x_n| = \sum c_n |\lim_p x^p_n| = \lim_p \sum c_n |x_n^p| \leq \sup_ p \|x^p\|.$$
	
	Note that $\|e_n\| = \alpha_n$, and $\|(1,1,1,\cdots)\|=\sup\{\sum_n c_n : c\in C\}\leq 1$. So if we call $e=(1,1,1,\ldots)$, then we have that $e\in X_\infty$. Let $X_0$ be the Banach sublattice of $X_\infty$ generated by $\{e_n : n\in \mathbb{N}\}$, and $X_1$ the sublattice generated by $\{e_n : n\in \mathbb{N}\}\cup \{e\}$. As an illustrating (extreme) example, if $\alpha_n=1$ for all $n$, then $\|x\| = \|x\|_\infty$, $X_\infty = \ell_\infty$, $X_0 = c_0$ and $X_1=c$. Given $x=(x_1,x_2,...)$, we define $T_n x = (0,0,\ldots,0,x_n,x_{n+1},\ldots)$. That is, $T_n(x)_k = 0$ if $k<n$, while $T_n(x)_k = x_k$ if $k\geq n$.\\
	
	\noindent
	\emph{Claim 0.} We have that $x\in X_0$ if and only if $\lim_n \|T_n x\| = 0$.
	
	\begin{proof}
		$[\Leftarrow]$ Note that $x-T_n(x)\in span\{e_1,\ldots,e_n\}\in X_0$. So $x= \lim_n (x-T_n(x))\in X_0$.
		$[\Rightarrow]$ The set $c_{00}$ of finitely supported vectors forms a vector lattice whose closure is $X_0$. So for each $\varepsilon$ there exists $x^{\varepsilon}$ supported below $n_\varepsilon$ such that $\|x-x^\varepsilon\|<\varepsilon$. But note that for $n>n^\varepsilon$, $|T_n(x)| \leq |x-x^\varepsilon|$. 	
	\end{proof}
	
	\noindent
	\emph{Claim 1.} $X_1 = X_0 + \mathbb{R}e$.
	
	\begin{proof}
		It is enough to check that if $x\in X_0 + e\mathbb{R}$ then $|x|\in X_0+e\mathbb{R}$. We know that there is $r\in\mathbb{R}$ such that $\lim_n\|T_n(x-re)\|=0$. It is enough to prove that $\lim_n\|T_n(|x|-|r|e)\| = 0$. Since $T_n$ is a lattice homomorphism $|T_n(|x|-|r|e)| \leq |T_n(x-re)|$ and therefore $\|T_n(|x|-|r|e)\| \leq \|T_n(x-re)\|$.
	\end{proof}
	
	Now note that, by Claim 0, $\mathcal{I} = \{A\subseteq\mathbb{N} : 1_A\in X_0\}$.\\

	\noindent
	\emph{Claim 2.} $X_0\cap \ell_\infty =  \left\{(x_n)\in \ell_\infty : \lim_{\mathcal{F}}x_n = 0\right\}$.

	\begin{proof}
		$[\subseteq]$ If $\lim_{\mathcal{F}}x_n \neq 0$, then there exists $\varepsilon>0$ such that $A=\{n : |x_n|>\varepsilon\}\not\in\mathcal{I}$. This means that $1_A\not\in X_0$, and at the same time $\varepsilon 1_A \leq |x|$. This is a contradiction, because it is clear from Claim 0 that if $x\in X_0$ and $|y|\leq |x|$, then $y\in X_0$. $[\supseteq]$ Suppose that $x$ is in the righthand set and we check that $x\in X_0$ using Claim 0. Fix $\varepsilon>0$. We know that $A=\{n : |x_n|>\varepsilon/2\}\in\mathcal{I}$, so $1_A\in X_0$. Thus, there exists $m_0$ such that for $m>m_0$ we have that $\|T_m 1_A\|<\frac{1}{2}\varepsilon\|x\|_\infty^{-1}$ and therefore
		$$\|T_m x\| \leq \|T_m (1_A x)\| + \|T_m(x-1_Ax)\| \leq \|x\|_\infty \|T_m 1_A\| + \|T_m(x-1_Ax)\|_\infty < \varepsilon.$$
	\end{proof}
	
	\noindent
	\emph{Claim 3.} $X_1\cap \ell_\infty =  \left\{(x_n)\in \ell_\infty : \lim_{\mathcal{F}}x_n \text{ exists}\right\}$. 
	\begin{proof}
		Follows from Claims 1 and 2.	
	\end{proof}

	It remains to show that	$K_e(X_1)$ is homeomorphic to the Stone space  of the algebra $\mathcal{B} = \mathcal{I}\cup \mathcal{F}$. This Stone space $L$ is obtained by gluing together all ultrafilters of $\betaN$ that contain $\mathcal{F}$, and so $C(L) \simeq \{(x_n)\in\ell_\infty : \lim_\mathcal{F}x_n \text{ exists}\}$. Note also that the ideal generated by $e$ inside $X_1$ equals $X_1\cap \ell_\infty$. So we just apply Claim 3.
\end{proof}

We refer to \cite{BFP,BF,Farah} for further information on examples and non-examples of non-pathological analytic $P$-ideals. We may just highlight two basic ones:

If we fix a sequence $(\alpha_n)$ of positive numbers with $\lim_n\alpha_n=0$, and we take
$$C=\{(c_n)\in c_{00}\cap B_{\ell_1}^+ : c_n\leq \alpha_n \text{ for all }n\},$$
then we obtain the summable ideal $\mathcal{I}(C) = \{A\subseteq \mathbb{N} : \sum_{n\in A}\alpha_n <\infty\}.$ This is because
$$\|T_m 1_A\| = \sup\left\{\sum_{n\in A, n\geq m} c_n |x_n| : c\in C\right\} = \min\left\{1,\sum_{n\in A,n\geq m}\alpha_n\right\}.$$

On the other hand, if we take $$C = \{(c_1=1/n,c_2=1/n,\ldots,c_n=1/n,0,0,0,\ldots) : n\in\mathbb{N}\}$$ then
$\mathcal{I}(C)$ is the ideal of statistically zero sets, because
$$\|T_m 1_A\| = \sup\left\{\sum_{n\in A, n\geq m} c_n |x_n| : c\in C\right\} = \sup_n\left\{\frac{|A|\cap\{m,m+1,\ldots,n\}}{n}\right\}.$$

\section{Open questions}

We have shown (see Section \ref{SectionAnalytic}) that sick compacta always admit strictly positive regular measures of countable type which are analytic, and satisfy several other properties. However, the following is still open:

\begin{prob}
	Find an intrinsic characterization of sick compacta.
\end{prob}

In connection with this we have a seemingly simpler question:

\begin{prob} 
	If $\mu$ is a strongly countably determined measure on a sick compactum (that is, a $G_\delta$-point in $P(K)$), is $\mu$ an analytic measure?
\end{prob}

We mentioned at the end of Section \ref{SectionFirstExamples} that the sick compactum $K(FBL(\N))$ contains every sick compactum, which coincides with the structure space $K(FBL(\N))$ of the free Banach lattice $FBL[\ell_1]$ generated by $\ell_1$ (see Corollary 2.9 in \cite{ART18}). With a bit more generality one can ask the following:

\begin{prob} 
	Can we give a description, in terms of more familiar spaces, of the structure spaces of some relevant Banach lattices, like $K(C(M,L_1))$, for the space of continuous functions from a metric compactum $M$ into $L_1$ (cf. \cite{LLOT19}), or $K(FBL[E])$, the free Banach lattice generated by a Banach space $E$?
\end{prob}

%Among nonmetrizable sick compacta we have seen $\beta\mathbb N$ plays a central role. This partly motivates the following.

\begin{prob} 
	Is there a nonseparable connected sick compactum?
\end{prob}

Recall that we defined s-sick algebras as those isomorphic to the clopen sets of a sick compactum, and characterized them in Theorem \ref{supremaofsequences}. In connection with this, and Stone duality, the following is a very natural question.

\begin{prob} 
	Is it true that if $K$ is sick, then the Stone space of the algebra of clopen subsets of $K$ is also sick?
\end{prob}

Motivated by Corollary \ref{corproductos} and Proposition \ref{propproductcountablyclopens}, we are lead to consider the following.

\begin{prob} 
	Is there a nonmetrizable product of two infinite compact spaces that is sick?
\end{prob}

Keeping in mind Proposition \ref{orthogonalideal} and its consequences, the ideals $J(\mathcal{B})$ associated to s-sick algebras have certain relevance. We thus propose the following.

\begin{prob} 
	Characterize the ideals $J(\mathcal{B})$ for s-sick algebras $\mathcal{B}$.
\end{prob}

In connection with the results of Section \ref{SectionPideals}, one is lead to wonder:

\begin{prob} 
	For what ideals $\mathcal{I}$ of subsets of $\mathbb{N}$ is the algebra generated by $\mathcal{I}$ s-sick? Out of the trivial case of maximal ideals, must such $\mathcal{I}$ be non-pathological analytic $P$-ideals?
\end{prob}

% ------------------------------------------------------------------------

\subsection*{Acknowledgment}
Research partially supported by Fundaci\'on S\'eneca - ACyT Regi\'on de Murcia grant 20797/PI/18 and Agencia Estatal de Investigación and EDRF/FEDER ``A way of making Europe" (MCIN/AEI/10.13039/501100011033) through grants MTM2017-86182-P and PID2021-122126NB-C32 (Avilés, Martínez-Cervantes and Rueda Zoca), PGC2018-093794-B-I00 (Rueda Zoca), PID2020-116398GB-I00 and CEX2019-000904-S (Tradacete). The research of A. Rueda Zoca was also supported by Junta de Andaluc\'ia Grants A-FQM-484-UGR18 and FQM-0185.

\end{document}